\documentclass[12pt,reqno]{article}

\usepackage[usenames]{color}
\usepackage{amssymb}
\usepackage{amsmath}
\usepackage{amsthm}
\usepackage{amsfonts}
\usepackage{amscd}
\usepackage{graphicx}

\usepackage[colorlinks=true,
linkcolor=webgreen,
filecolor=webbrown,
citecolor=webgreen]{hyperref}

\definecolor{webgreen}{rgb}{0,.5,0}
\definecolor{webbrown}{rgb}{.6,0,0}

\usepackage{color}
\usepackage{fullpage}
\usepackage{float}

\usepackage{graphics}
\usepackage{latexsym}
\usepackage{epsf}
\usepackage{breakurl}

\newcommand{\seqnum}[1]{\href{https://oeis.org/#1}{\rm \underline{#1}}}

\begin{document}


\theoremstyle{plain}
\newtheorem{theorem}{Theorem}
\newtheorem{corollary}[theorem]{Corollary}
\newtheorem{lemma}[theorem]{Lemma}
\newtheorem{proposition}[theorem]{Proposition}

\theoremstyle{definition}
\newtheorem{definition}[theorem]{Definition}
\newtheorem{example}[theorem]{Example}
\newtheorem{conjecture}[theorem]{Conjecture}

\theoremstyle{remark}
\newtheorem{remark}[theorem]{Remark}

\begin{center}
\vskip 1cm{\LARGE\bf 
Linear $k$-Chord Diagrams
}
\vskip 1cm 
\large
Donovan Young\\
St Albans, Hertfordshire AL1 4SZ \\
United Kingdom\\
\href{mailto:donovan.m.young@gmail.com}{\tt donovan.m.young@gmail.com} \\
\end{center}

\vskip .2 in

\begin{abstract}
We generalize the notion of linear chord diagrams to the case of
matched sets of size $k$, which we call $k$-chord diagrams. We provide
formal generating functions and recurrence relations enumerating these
$k$-chord diagrams by the number of short chords, where the latter is
defined as all members of the matched set being adjacent, and is the
generalization of a short chord or loop in a linear chord diagram. We
also enumerate $k$-chord diagrams by the number of connected
components built from short chords and provide the associated
generating functions in this case. We show that the distributions of
short chords and connected components are asymptotically Poisson, and
provide the associated means. Finally, we provide recurrence relations
enumerating non-crossing $k$-chord diagrams by the number of short
chords, generalising the Narayana numbers, and establish asymptotic
normality, providing the associated means and variances. Applications
to generalized games of memory are also discussed.
\end{abstract}

\section{Introduction and basic notions}
\label{Seckchords}

A chord diagram is a set of $n$ chords drawn between $2n$ distinct
points (which we call {\it vertices}) on a circle, such that each
vertex participates in exactly one chord. In a linear chord diagram
the circle is replaced by a linear arrangement of $2n$ vertices, and
the chords are represented as arcs. The study of chord diagrams has a
long and varied history. Early results by Touchard \cite{T} and
Riordan \cite{R} studied the number of crossings\footnote{Cf.\ Pilaud
  and Ru\'{e} \cite{PR} for a modern approach and further
  developments, and also Krasko and Omelchenko \cite{KO} for a more
  complete list of references.}. Kreweras and Poupard~\cite{KP}
studied linear chord diagrams. One of their results is the enumeration
of {\it short chords}, which are defined as chords formed on an
adjacent pair of vertices\footnote{Short chords are also called
  ``loops'' as in~\cite{KO}, or originally ``paires courtes''
  in~\cite{KP}.}.  They provided recurrence relations and closed form
expressions for the number of configurations with exactly $\ell$ short
chords. They also showed that the mean number of short chords is $1$,
which implies that the total number of short chords is equinumerous
with the total number of linear chord diagrams,
cf.\ \cite{CK}. Kreweras and Poupard~\cite{KP} showed further that all
higher factorial moments of the distribution approach $1$ in the
$n\to\infty$ limit, thus establishing the Poisson nature of the
asymptotic distribution.
\begin{figure}[H]
\begin{center}
\includegraphics[bb=0 0 315 68, height=0.5in]{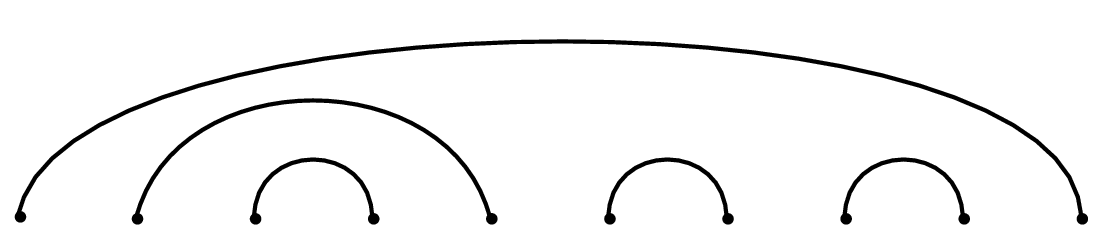}
\end{center}
\caption{A configuration of non-crossing chords.}
\label{Figlibres}
\end{figure}
Kreweras and Poupard~\cite{KP} also introduced the concept
of a ``free pair''\footnote{They refer to this as ``une paire
  libre''.}. A free pair is a chord which is not crossed by any other
chord, and only contains other free pairs, see
Figure~\ref{Figlibres}. We will refer to these as {\it non-crossing
  chords}. In Section~\ref{Secnc} we will review Kreweras and
Poupard's result that the number of linear chord diagrams consisting
entirely of non-crossing chords is in bijection with Dyck paths and is
hence counted by the Catalan numbers, and that the number of them with
exactly $\ell$ short chords is counted by the Narayana numbers.

In this paper, we consider a generalization of linear chord diagrams
which we refer\footnote{Pilaud and Ru\'{e} \cite{PR} call an
  essentially equivalent object a ``hyperchord diagram''.} to as
linear {\it $k$-chord diagrams}. In a $k$-chord diagram the basic
matching represented by a chord is enlarged from a pair of vertices to
a set of $k$ vertices, i.e., a {\it $k$-chord}, see
Figure~\ref{Figkchord}; a usual chord diagram corresponds to the case
$k=2$.
\begin{figure}[H]
\begin{center}
\includegraphics[bb=0 0 656 60, height=0.5in]{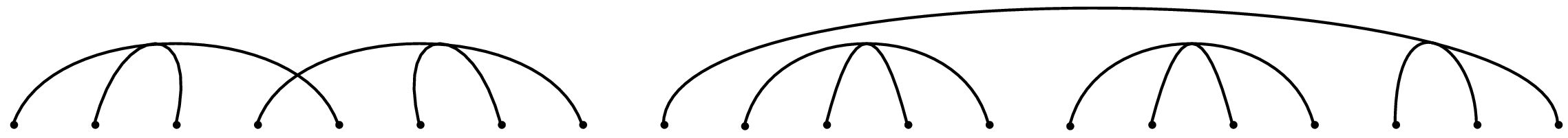}
\end{center}
\caption{A linear $k$-chord diagram with $k=4$ and $n=5$. This is a
  configuration with 2 short chords, 3 non-crossing chords, and 1
  connected component (formed by the two adjacent short chords).}
\label{Figkchord}
\end{figure}
The concept of a short chord generalizes directly, as follows:
\begin{definition}
  A \emph{short chord} is a $k$-chord formed on a set of
    adjacent vertices.
\end{definition}
We will also be concerned with the notion of a \emph{connected
  component} of short chords, see Figure~\ref{Figkchord}.
\begin{definition}\label{Defcomp}
  A \emph{connected component} in a linear $k$-chord diagram is a set
  of adjacent short chords.
\end{definition}
The concept of a non-crossing chord can be generalized to
$k$-chord diagrams as follows:
\begin{definition} 
  A \emph{non-crossing chord} is a $k$-chord which is not
    crossed by any other $k$-chord and only contains other non-crossing chords.
\end{definition}
Before presenting the main results of the paper, it is useful
to establish a few basic facts.
\begin{proposition}\label{Nnk}
The number of linear $k$-chord diagrams of length $kn$ is given by
\begin{equation}\nonumber
{\cal N}_{k,n}=\frac{(kn)!}{(k!)^n \,n!}.  
\end{equation}
\end{proposition}
\begin{proposition}\label{Prpzero}
The number of number of linear $k$-chord diagrams of length $kn$
without any short chords is given by
\begin{equation}\nonumber
\sum_{j = 0}^{n} (-1)^j\,{\cal N}_{k,n-j}\,\rho_j ,
\end{equation}
where $\rho_j$ represents the number of ways of choosing $j$ disjoint
sub-paths, each with $k$ vertices, from the path of length $kn$. We
define $\rho_0$ to be 1.
\end{proposition}
\begin{proof}
The proof proceeds via inclusion-exclusion. We note that for each of
the $\rho_j$ choices of $j$ sub-paths on which to place $j$ short
chords, there remains ${\cal N}_{k,n-j}$ configurations of the
remaining $n-j$ $k$-chords. There will be some number of
configurations among these ${\cal N}_{k,n-j}$ with exactly $q$ short
chords. Then ${\cal N}_{k,n-j}\,\rho_j$ counts the configurations with
exactly $q+j$ short chords ${q+j\choose j}$ times. Let $N(q)$ be the
number of configurations with exactly $q$ short chords, we therefore
have that
$$\sum_{j=0}^{n} (-1)^j\,{\cal N}_{k,n-j} \,\rho_j
=\sum_{j=0}^{n} (-1)^j
\sum_{q=0}^{n-j} {q+j\choose j} N(q+j)$$
$$=N(0) + \sum_{q+j=1}^n N(q+j) \sum_{j=0}^{q+j} (-1)^j 
{q+j\choose j},$$
and so all but the zero-short-chord configurations cancel.
\end{proof}
\begin{theorem}\label{Thmmean}
The mean number of short chords in a linear $k$-chord diagram of
length $kn$ is given by
\begin{equation}\nonumber
{kn \choose k}^{-1}n\left(kn-(k-1)\right).
\end{equation}
\end{theorem}
\begin{proof}
The proof proceeds through the linearity of expectation. Let the
random variable $X_j$ take the value $1$ when the $j^{\text{th}}$
consecutive set of $k$ vertices forms a short chord and $0$
otherwise. Once a short chord is thusly placed, by
Proposition~\ref{Nnk}, there are ${\cal N}_{k,n-1}$ ways of placing
the remaining $k$-chords on the $kn-k$ remaining vertices. Thus
$E(X_j) = {\cal N}_{k,n-1}/{\cal N}_{k,n}$. We therefore have that
$E(\sum_{j} X_j) = \sum_{j} E(X_j) = \left(kn-(k-1)\right) \, {\cal
  N}_{k,n-1}/{\cal N}_{k,n}$, where we have used the fact that there
are $\left(kn-(k-1)\right)$ consecutive sets of $k$ vertices.
\end{proof}
The result of Theorem~\ref{Thmmean} shows that the total
number of $k$-chords in all diagrams of a fixed length does not equal
the total number of diagrams, unless $k=2$. In the limit of long
diagrams, the mean scales as $\lim_{n\to\infty} {kn\choose
  k}^{-1}n\,(kn-(k-1)) = k!\,k^{1-k}\,n^{2-k},$ and is hence generally
suppressed for large $n$.

\section{Enumeration by short chords}

We begin by enumerating configurations by number of short chords.
\begin{theorem}\label{ThmCF}
  The number $d_{n,\ell}$ of linear $k$-chord diagrams of length $kn$ with
  exactly $\ell$ short chords is
  \begin{equation}\nonumber
    d_{n,\ell} = \frac{1}{\ell!}\sum_{j=\ell}^n \frac{(k(n-j)+j)!\,(-1)^{j-\ell}}
    {(k!)^{n-j}\,(n-j)!\,(j-\ell)!}.
    \end{equation}
\end{theorem}
\begin{proof}
  This follows from direct calculation from Lemma~\ref{Lemnearly}.
\end{proof}

In what follows, OEIS denotes the {\it On-Line Encyclopedia of Integer
Sequences} \cite{Sloane}.
\begin{table}[H]
\begin{center}
  \begin{tabular}{c|lllllll}
  $n$ \textbackslash$\ell$& 0& 1& 2& 3& 4& 5& 6\\
  \hline  
  1& 0& 1 \\
  2& 7& 2& 1\\
  3& 219& 53& 7& 1\\
  4& 12861& 2296& 226& 16& 1\\
  5& 1215794& 171785& 13080& 710& 30& 1\\
  6& 169509845& 19796274& 1228655& 53740& 1835& 50& 1\\
\end{tabular}
\end{center}
\caption{The number $d_{n,\ell}$ of linear $k$-chord diagrams of
  length $kn$ with exactly $\ell$ short chords, for the case
  $k=3$. OEIS sequence \seqnum{A334056}; the $k=4$ case is
  \seqnum{A334057} and the $k=5$ case is \seqnum{A334058}.}
\end{table}

\begin{lemma}\label{Lemchpath}
The number of ways of choosing $j$ pairwise non-overlapping subpaths,
each of length $k$, from the path of length $\ell$ is
\begin{equation}\nonumber
  {\ell - j(k-1)\choose j}.
  \end{equation}
\end{lemma}
\begin{proof}
For each of the $j$ subpaths, collapse the vertices of that subpath
onto the left-most vertex, and mark it. We are left with $\ell-j(k-1)$
vertices, $j$ of which are marked. Thus there are ${\ell - j(k-1)\choose j}$
choices for the positions of the marked vertices.
\end{proof}

\begin{lemma}\label{Lematleast}
The number of linear $k$-chord diagrams of length $kn$ with at least
$j$ short chords is given by
\begin{equation}\nonumber
  \frac{(k(n-j))!}{(k!)^{n-j} \,(n-j)!} {kn - j(k-1)\choose j}.
  \end{equation}
\end{lemma}
\begin{proof}
We choose $j$ pairwise non-overlapping subpaths, enumerated according
to Lemma~\ref{Lemchpath}, to place the short chords upon. By
Proposition~\ref{Nnk}, for each such choice we have ${\cal N}_{k,n-j}$
ways of placing the remaining $k$-chords.
\end{proof}

\begin{lemma}\label{Lemnearly}
The number of linear $k$-chord diagrams of length $kn$ with exactly
$\ell$ short chords is given by
\begin{equation}\nonumber
[z^\ell] \sum_{j=0}^n
  \frac{(k(n-j))!}{(k!)^{n-j} \,(n-j)!} {kn - j(k-1)\choose j} (z-1)^j.
\end{equation}
\end{lemma}
\begin{proof}
This follows from inclusion-exclusion, cf.\ \cite[p.\ 112]{Wilf}.  
\end{proof}

\subsection{Two recurrence relations}

Kreweras and Poupard~\cite{KP} gave a recurrence relation for the
$d_{n,\ell}$ for the case of $k=2$,$$\ell\, d_{n,\ell} = (2n -
\ell)\,d_{n-1,\ell-1} + \ell\,d_{n-1,\ell}\quad\Leftrightarrow\quad
k=2.$$ They obtained this by considering the removal of a short chord at a
given position in the chord diagram. This short chord is either nested
directly inside another chord, such that its removal does not change
the number of short chords, or it is not. A sum over all possible
positions of the given short chord then results in the recurrence
relation. We now give a generalization of this recurrence relation to
the case of general $k$. Rather than attempt to repeat Kreweras and
Poupard's arguments, which should be possible but rather complicated
due to dealing with the ends of the diagram, we prove our recurrence
relation directly from Theorem~\ref{ThmCF}.

\begin{theorem}\label{ThmKP1}
  The numbers $d_{n,\ell}$ satisfy the following recurrence relation
  \begin{equation}\nonumber
    \ell\, d_{n,\ell} = (kn - \ell(k-1))\,d_{n-1,\ell-1} + \ell(k-1)\,d_{n-1,\ell}.
    \end{equation}
\end{theorem}
\begin{proof}
We use the result of Theorem~\ref{ThmCF} and consider $
d_{n-1,\ell-1}-d_{n-1,\ell}$. In each of the two summands, we
shift the summation variable $j\to j-1$. This procedure results in a
relative sign difference between the summands. In $d_{n-1,\ell-1}$
the sum begins at $j=\ell$ as before, whereas in $d_{n-1,\ell}$, it
now begins at $j=\ell+1$. We proceed by stripping off the $j=\ell$ term from $d_{n-1,\ell-1}$
\begin{align}\nonumber
    &d_{n-1,\ell-1}-d_{n-1,\ell} =
\frac{(k(n-\ell)+\ell-1)!}
    {(k!)^{n-\ell}\,(n-\ell)!}
    \frac{1}{(\ell-1)!} \\\nonumber
    &\qquad\qquad+\sum_{j=\ell+1}^n\frac{(k(n-j)+j-1)!\,(-1)^{j-\ell}}
    {(k!)^{n-j}\,(n-j)!}
    \left(
    \frac{1}{(j-\ell)!(\ell-1)!} + \frac{1}{(j-\ell-1)!\ell!} 
    \right)\\\nonumber
     &=  \frac{(k(n-\ell)+\ell-1)!}
    {(k!)^{n-\ell}\,(n-\ell)!}
    \frac{1}{(\ell-1)!}
    + \sum_{j=\ell+1}^n\frac{(k(n-j)+j-1)!\,(-1)^{j-\ell}}
    {(k!)^{n-j}\,(n-j)!}
    \left(
    \frac{j/\ell}{(j-\ell)!(\ell-1)!} 
    \right).
\end{align}
It follows that
\begin{align}\nonumber
    &kn\,d_{n-1,\ell-1} -\ell(k-1) \left( d_{n-1,\ell-1}-d_{n-1,\ell} \right)\\\nonumber
    &\qquad= \frac{(k(n-\ell)+\ell-1)!}
    {(k!)^{n-\ell}\,(n-\ell)!}
    \frac{(kn-\ell(k-1))}{(\ell-1)!}\\\nonumber
    &\qquad\qquad\qquad\qquad\qquad\qquad
    +\sum_{j=\ell+1}^n\frac{(k(n-j)+j-1)!\,(-1)^{j-\ell}}
    {(k!)^{n-j}\,(n-j)!}
    \left(
    \frac{kn-j(k-1)}{(j-\ell)!(\ell-1)!} 
    \right)\\\nonumber
   &\qquad= \frac{1}{(\ell-1)!}\sum_{j=\ell}^n \frac{(k(n-j)+j)!\,(-1)^{j-\ell}}
    {(k!)^{n-j}\,(n-j)!\,(j-\ell)!} = \ell\, d_{n,\ell}.
 \end{align}
\end{proof}

Krasko and Omelchenko~\cite{KO} gave another recurrence relation for
the $d_{n,\ell}$ in the case of $k=2$, $$ d_{n+1,\ell} = d_{n,\ell-1} +
(2n-\ell)d_{n,\ell} + (\ell+1)d_{n,\ell+1}\quad\Leftrightarrow\quad
k=2.$$ This relation is found by considering the addition of a new
chord, with one end anchored outside the existing vertices. The other
end of the chord could be placed in various positions: adjacent to the
anchored end hence forming an external short chord, in a location
which disturbs no existing short chord, or within an existing short
chord and hence breaking it up. These cases account, respectively, for
the three terms in the recurrence relation. To generalize this to a
$k$-chord diagram we must account for a larger variety of cases.

\begin{theorem}\label{ThmKP2}
  The numbers $d_{n,\ell}$ satisfy the following recurrence relation
  \begin{equation}\nonumber
    d_{n+1,\ell} = d_{n,\ell-1} + d_{n,\ell} \sum_{h=1}^{k-1}
    {kn-(k-1)\ell+k-h-1 \choose k-h}
    + \sum_{p=1}^{k-1} C_{n,\ell,p,k} \,d_{n,\ell+p} ,
  \end{equation}
  where the $C_{n,\ell,p,k}$ are given by
  \begin{equation}\nonumber
    C_{n,\ell,p,k} = \sum_{h=1}^{k-p}\sum_{f=0}^{k-p-h}
    {kn-(k-1)(\ell+p)+f-1 \choose f}
    [x^{k-h-f}y^p]\left(1+y-y(1-x)^{1-k}\right)^{-\ell-1}.
   \end{equation}
\end{theorem}
\begin{figure}[H]
\begin{center}
\includegraphics[bb=0 0 440 68, height=0.5in]{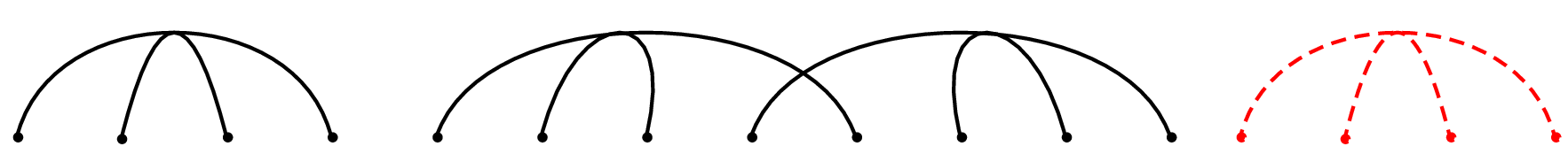}
\end{center}
\caption{The first term in the recurrence relation in
  Theorem~\ref{ThmKP2}. A new $k$-chord (shown in
  red dashes) is added as a short chord to the end of an existing
  configuration.}
\label{FigKP2addone}
\end{figure}

\begin{proof}
The recurrence relation arises from the addition of a new $k$-chord, at
least one strand of which is anchored to the right of all existing
vertices, to a configuration of length $kn$. One of the ways we could
add these vertices is as a short chord placed on the end of the existing
configuration, as in Figure~\ref{FigKP2addone}. This accounts for the
first term in the recurrence relation.

We now consider the various positions which the strands of the new
$k$-chord could occupy. If they remain to the right of all existing
vertices, we consider them to be at {\it home}. If a strand is not at
home, but is placed such that it does not disturb any existing
short chord, we say that it is in the {\it forest}. If the existing
configuration has $\ell+p$ short chords, there are $kn-(k-1)(\ell+p)$
forest positions.

The second term in the recurrence relation arises from leaving $h\geq 1$
strands at home, and placing the remaining $f=k-h$ strands into
forest positions only, i.e., no strand which is not at home is placed
within an existing short chord. The number of ways of accomplishing
this is the same as the number of ways of placing $f=k-h$ identical
balls into $kn-(k-1)\ell$ distinguishable bins.

The third term in the recurrence relation accounts for the general
case: $h\geq 1$ strands are left at home, $f\geq 0$ occupy forest
positions, and the remainder disturb existing short chords. An example
is shown in Figure~\ref{FigKP2}. The result follows from the
preceding treatment of forest positions (there are now
$kn-(k-1)(\ell+p)$ of these) in conjunction with Lemma~\ref{LemKP2},
where the $j$ balls represent the $k-h-f$ strands which disturb
short chords, and the $p$ bins represent the $p$ disturbed short chords,
each of which can be disturbed in up to $k-1$ positions.
\end{proof}
\begin{figure}[H]
\begin{center}
\includegraphics[bb=0 0 440 68, height=0.5in]{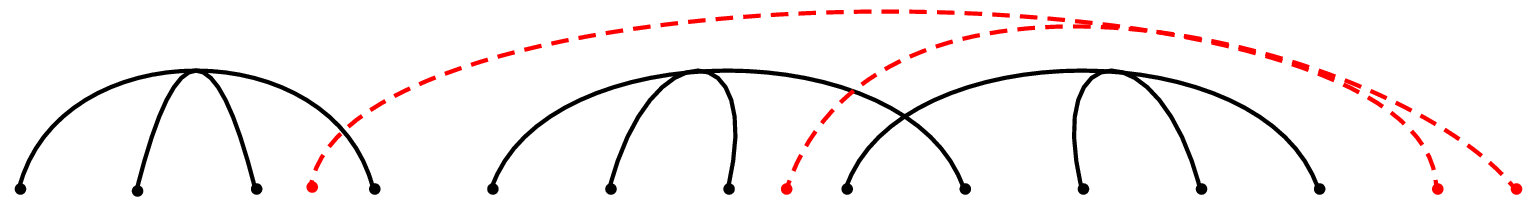}
\end{center}
\caption{A general configuration from Theorem~\ref{ThmKP2}, with
  $k=4$, $n=3$, $\ell=0$, $h=2$, $f=1$, and $p=1$.}
\label{FigKP2}
\end{figure}
\begin{example}
  The recurrence relation for the case $k=3$ is given by
  \begin{align}\nonumber
      d_{n+1,\ell} &= d_{n,\ell-1}\\\nonumber &
      + \frac{(3n-2\ell+3)(3n-2\ell)}{2}\,d_{n,\ell} + (\ell+1) (6
      n-4 \ell+1)\,d_{n,\ell+1} + 2 (\ell+1) (\ell+2)\, d_{n,\ell+2}.
\end{align}
\end{example}

\begin{lemma}\label{LemKP2}
The number of ways of placing $j$ identical balls into a selection of
$p$, out of $\ell+p$ distinguishable bins, each containing $k-1$
distinguishable sub-bins, such that no bin is empty (though any
sub-bin could be), is
\begin{equation}\nonumber
  [x^{j}y^p]\left(1+y-y(1-x)^{1-k}\right)^{-\ell-1}.
\end{equation}
\end{lemma}
\begin{proof}
We first note that there are $\ell+p\choose p$ ways to choose the $p$
bins. To enumerate the possible ways of filling the sub-bins we
consider the weak compositions of $1,2,3,\ldots$ into $k-1$ parts. For
a given bin, let $f(x)$ be the generating function such that
$[x^m]f(x)$ counts the number of ways of placing $m$ balls into the
$k-1$ sub-bins. We have that
\begin{equation}\nonumber
  f(x) = {k-1\choose k-2} x + {k\choose k-2} x^2+ {k+1\choose k-2} x^3 +\cdots
  =(1-x)^{1-k}-1.
\end{equation}
To now account for $p$ bins, we take $f(x)^p$, and note that
\begin{equation}\nonumber
\sum_{p} {\ell+p\choose p} f(x)^p y^p =  \left(1+y-y(1-x)^{1-k}\right)^{-\ell-1}.
\end{equation}
\end{proof}

\section{Generating functions}
\label{SecGFs}

In this section we establish formal generating functions enumerating
configurations firstly by number of short chords
(Theorem~\ref{Thmbypolys}),
\begin{equation}\nonumber
    F_k(w,z)=\sum_{n\geq 0}\sum_{\ell=0}^nd_{n,\ell}\, w^n z^\ell=\sum_{j\geq 0}\frac{(kj)!}{j!(k!)^j}
    \frac{w^j}{\left(1+w(1-z)\right)^{kj+1}},
\end{equation}
where the power of $w$ corresponds to $n$ and the power of $z$
corresponds to the number of short chords, and secondly by the number
of connected components (Theorem~\ref{Thmbycomp})
 \begin{equation}\nonumber
 C_k(y,z)=\sum_{n,q\geq 0}c_{n,q}\, y^n z^q=   
  \sum_{j\geq 0}\frac{(kj)!}{j!(k!)^j}\,
  y^j\left(\frac{1-y(1-z)}{1-y^2(1-z)}\right)^{kj+1},
 \end{equation}
 where the power of $y$ corresponds to $n$ and the power of $z$
 corresponds to the number of connected components. These generating
 functions are not convergent, but nevertheless provide compact
 expressions for the numbers they count.

\subsection{Counting by number of short chords}
  
\begin{proposition}\label{PrpL}
The generating function which counts the number of ways of choosing
$j$ pairwise non-overlapping subpaths, each of length $k$, from a path
of length $\ell$ is
\begin{equation}\nonumber
L_k(x,y) = \frac{1}{1-y(1+x^ky^{k-1})},
\end{equation}
where the power of $x$ corresponds to the total number of vertices in
all chosen sub-paths, and the power of $y$ corresponds to the length
of the path.
\end{proposition}
\begin{proof}
By direct expansion $L_k(x,y) = \sum_{p\geq 0}\sum_{j=0}^p {p\choose
  j}x^{kj} \,y^{(k-1)j+p}$, so $[x^{kj}y^{\ell}]L_k(x,y) = {\ell -
  j(k-1)\choose j}$. By Lemma~\ref{Lemchpath}, the proposition is
proven.
\end{proof}

\begin{lemma}\label{Lemmethod}
  The generating function which counts the number of linear $k$-chord
  diagrams of length $kn$ with at least $j$ short chords is given by
  \begin{equation}\nonumber
    N_k(w^k,z^k)=\int_0^\infty dt \,e^{-t}\,\frac{1}{2\pi i}
    \oint_{|x|=\epsilon} \frac{dx}{x}\, e^{x^k/k!} \,
    L_k\left(\frac{zx}{t},\frac{wt}{x}\right),
    \end{equation}
where $[w^n z^j] \,N_k(w,z)$ is the number of linear $k$-chord diagrams of
length $kn$ with at least $j$ short chords.
\end{lemma}
\begin{proof}
  We begin by noting that
  \begin{align}\nonumber
    [w^{kn} z^{kj}] \int_0^\infty dt \,e^{-t}\,
    L_k\left(\frac{zx}{t},\frac{wt}{x}\right) &=\int_0^\infty dt
    \,e^{-t}\,t^{kn-kj}{kn - j(k-1)\choose j} x^{kj-kn}\\\nonumber
    &= (kn-kj)! {kn
      - j(k-1)\choose j} x^{kj-kn},
  \end{align}
  where we have made use of Lemma~\ref{Lemchpath}. We now note that
  \begin{equation}\nonumber
    [x^{kn-kj}]e^{x^k/k!} =\frac{1}{(k!)^{n-j}\,(n-j)!},
  \end{equation}
  and so the contour integral selects precisely this term in the
  expansion of $e^{x^k/k!}$. We therefore have that
\begin{align}\nonumber
     [w^{kn} z^{kj}]\int_0^\infty dt \,e^{-t}\,\frac{1}{2\pi i}
    \oint_{|x|=\epsilon} \frac{dx}{x}\, e^{x^k/k!} \,
    L_k\left(\frac{zx}{t},\frac{wt}{x}\right)  
    =\frac{(kn-kj)!}{(k!)^{n-j}\,(n-j)!} {kn
      - j(k-1)\choose j},
  \end{align}
  and by Lemma~\ref{Lematleast} the proof is complete.
\end{proof}

\begin{theorem}\label{Thmbypolys}
The generating function which counts the number of linear $k$-chord diagrams
of length $kn$ with exactly $\ell$ short chords is given by
\begin{equation}\nonumber
  F_k(w,z)=
  \sum_{j\geq 0}\frac{(kj)!}{j!(k!)^j}
  \frac{w^j}{\left(1+w(1-z)\right)^{kj+1}},
\end{equation}
where the power of $w$ corresponds to $n$ and the power of $z$
corresponds to $\ell$.
\end{theorem}
\begin{proof}
  We use Lemma~\ref{Lemmethod} to find the generating function for at
  least $j$ short chords, and then replacing $z\to z-1$ at the end, as
  per Lemma~\ref{Lemnearly}, yields the desired result. We note that
  \begin{equation}\nonumber
  \frac{1}{2\pi i}
  \oint_{|x|=\epsilon} \frac{dx}{x}\, e^{x^k/k!}
  L_k\left(\frac{zx}{t},\frac{wt}{x}\right)  =
   \frac{1}{2\pi i}
   \oint_{|x|=\epsilon}dx\,  \frac{e^{x^k/k!}}{x\left(1-(wz)^k\right)-tw},
  \end{equation}
  where we have made use of Proposition~\ref{PrpL}. We proceed by
  evaluating the residue at $x=tw/(1-(wz)^k)$
  \begin{align}\nonumber
    N_k(w^k,z^k) &= \int_0^\infty dt \,e^{-t}\,
    \frac{1}{1-(wz)^k}
    \exp\left(\frac{t^kw^k}{k!\left(1-(wz)^k\right)^k}\right)\\\nonumber
     &= \int_0^\infty dt \,e^{-t}\,
    \frac{1}{1-(wz)^k}
    \sum_{j\geq 0}\frac{1}{j!}\left(\frac{t^kw^k}{k!\left(1-(wz)^k\right)^k}\right)^j,
    \end{align}
  where the integration over $t$ is understood to be performed
  term-by-term in the expansion of the exponential, thus yielding a
  factor of $(kj)!$. Note that $w$ and $z$ appear uniformly with
  exponent $k$; we can therefore remove the exponents by considering
  $N_k(w,z)$ in place of $N_k(w^k,z^k)$. Finally, we have that
  \begin{equation}\nonumber
    F_k(w,z)=N_k(w,z-1) =  \sum_{j\geq 0}\frac{(kj)!}{j!(k!)^j}
  \frac{w^j}{\left(1+w(1-z)\right)^{kj+1}}.
     \end{equation}
\end{proof}

\subsection{Counting by number of connected components}

We now turn our attention to counting configurations by the number of
connected components, as defined in Definition~\ref{Defcomp}. Let
there be $q$ connected components; it is clear that there are
therefore at least $q-1$ and at most $q+1$ regions devoid of short
chords, depending on whether the first and last vertices are occupied
by short chords or not. We will require the number $\rho_j$ of ways of
choosing $j$ non-overlapping sub-paths from this disjoint collection of
regions. Proposition~\ref{Prpzero} will then count the number of
``zero-short-chord'' configurations on these regions.

\begin{lemma}
The numbers which count $\rho_j$ (cf.\  Proposition~\ref{Prpzero})
for the case of a single component of size $km$, where $m\geq 1$, are
denoted $\rho^{(1)}_j(m)$. We have that
\begin{equation}\nonumber
  \rho^{(1)}_j(m) = [x^{kj}y^{kn-km}]L_k(x,y)^2.
\end{equation}
\end{lemma}
\begin{proof}
The result follows from Proposition~\ref{PrpL}. The sum over the
position of the connected component is accounted for through the
symbolic method.
\end{proof}
Each further component which is added produces a new region, which,
rather than being bounded on one side by the ends of the $k$-chord diagram, is
bounded by two connected components. It is clear that these new
regions must not be allowed to have zero length; this is accomplished
by subtracting 1 from $L_k(x,y)$.
\begin{lemma}\label{Lemroq}
The numbers $\rho^{(q)}_j(m_1,\ldots,m_q)$ which count $\rho_j$ for the case of $q$
connected components, of sizes $km_1,\ldots, km_q$, summed over
positions, is
\begin{equation}\nonumber
\rho^{(q)}_j(m_1,\ldots,m_q) = [x^{kj}y^{kn-k\sum m_p}]\, L_k(x,y)^2 
\left(L_k(x,y) - 1\right)^{q-1}.
\end{equation}
\end{lemma}
\begin{proof}
This follows from the symbolic method.
\end{proof}

\begin{lemma}\label{Lemmethod2}
 \begin{align}\nonumber
      \sum_{j=0}^{\tilde n} (-1)^j &{\cal N}_{k,\tilde n - j}\,
      \rho^{(q)}_j(m_1,\ldots,m_q)\\\nonumber
      &=[y^{k\tilde n}]\int_0^\infty dt
      \,e^{-t}\,\frac{1}{2\pi i} \oint_{|x|=\epsilon} \frac{dx}{x}\,
      e^{x^k/k!} \,  L_k\left( \frac{xe^{i\pi/k}}{t},\frac{yt}{x}\right)^2
      \left( L_k\left(\frac{xe^{i\pi/k}}{t},\frac{yt}{x}\right) -
      1\right)^{q-1},
    \end{align}
where $\tilde n = n - \sum m_p$.
\end{lemma}

\begin{proof}
  The scaling of the $x$ and $y$ variables, together with the result
  of Lemma~\ref{Lemroq}, imply that
 \begin{align}\nonumber
    [y^{k\tilde n}] L_k\left( \frac{xe^{i\pi/k}}{t},\frac{yt}{x}\right)^2&
      \left( L_k\left(\frac{xe^{i\pi/k}}{t},\frac{yt}{x}\right) -
      1\right)^{q-1}\\\nonumber
      &= \sum_{j=0}^{\tilde n}(-1)^j x^{kj-k\tilde n}\,t^{k\tilde n - kj}
      \rho^{(q)}_j(m_1,\ldots,m_q).
       \end{align}
  Note that the integration over $t$ then enacts the replacement
  $t^{k\tilde n - kj}\to (k\tilde n - kj)!$. We now consider
 \begin{align}\nonumber
      [x^0] \, e^{x^k/k!} \, \sum_{j=0}^{\tilde n}(-1)^j x^{kj-k\tilde n}\,(k\tilde
      n - kj)!  \rho^{(q)}_j(m_1,\ldots,m_q) \\\nonumber =[x^0] \, \sum_{\ell}
      \frac{x^{k\ell}}{(k!)^\ell\ell!}  \, \sum_{j=0}^{\tilde n}(-1)^j
      x^{kj-k\tilde n}\, \rho^{(q)}_j(m_1,\ldots,m_q)\\\nonumber
      =\sum_{j=0}^{\tilde n} (-1)^j
      \frac{(k\tilde n - kj)!}{(k!)^{\tilde n - j}(\tilde n - j)!}
      \rho^{(q)}_j(m_1,\ldots,m_q),
       \end{align}
  and so the contour integral in $x$ picks out precisely this expression.
\end{proof}

\begin{table}[H]
\begin{center}
  \begin{tabular}{c|lllllll}
  $n$ \textbackslash$q$& 0& 1& 2& 3& 4\\
  \hline  
1& 0& 1\\
2& 7& 3\\
3& 219& 56& 5\\
4& 12861& 2352& 183& 4\\
5& 1215794& 174137& 11145& 323& 1\\
6& 169509845& 19970411& 1078977& 30833& 334\\
\end{tabular}
\end{center}
\caption{The numbers $c_{n,q}$ of linear $k$-chord diagrams of length
  $kn$ with exactly $q$ connected components for the case $k=3$. OEIS
  sequence \seqnum{A334060}; the $k=2$ case is \seqnum{A334059} and
  the $k=4$ case is \seqnum{A334061}.}
\end{table}

\begin{theorem}\label{Thmbycomp}
  The numbers $c_{n,q}=[y^nz^q]C_k(y,z)$ count the number of $k$-chord
  diagrams of length $kn$ with exactly $q$ connected components, where
 \begin{equation}\nonumber
 C_k(y,z)=   
  \sum_{j\geq 0}\frac{(kj)!}{j!(k!)^j}\,
  y^j\left(\frac{1-y(1-z)}{1-y^2(1-z)}\right)^{kj+1}.
\end{equation}
\end{theorem}
  \begin{proof}
We now make use of Lemma~\ref{Lemmethod2} and sum over the sizes
$m_1,\ldots,m_q$ of the connected components
   \begin{align}\nonumber
      &C_k(y^k,z)-C_k(y^k,0)
      =   \sum_{q\geq 1} z^q \sum_{\{m_p\geq 1\}} y^{k\sum m_p} \sum_{j=0}^{\tilde n}
(-1)^j {\cal N}_{k,\tilde n - j}\, \rho^{(q)}_j(m_1,\ldots,m_q)\\\nonumber
&=\sum_{q\geq 1} z^q \left(\frac{y^k}{1-y^k}\right)^q
\int_0^\infty dt
      \,e^{-t}\,\frac{1}{2\pi i} \oint_{|x|=\epsilon} \frac{dx}{x}\,
      e^{x^k/k!} \,  L_k\left( \frac{xe^{i\pi/k}}{t},\frac{yt}{x}\right)^2
      \left( L_k\left(\frac{xe^{i\pi/k}}{t},\frac{yt}{x}\right) -
      1\right)^{q-1}\\\nonumber
&=\int_0^\infty dt
      \,e^{-t}\,\frac{1}{2\pi i} \oint_{|x|=\epsilon} \frac{dx}{x}\,
      e^{x^k/k!} \,\sum_{q\geq 1}
      \frac{-x^2(zy^k)^q(xy^k-ty)^{q-1}}{(y^k-1)^q(x-ty+xy^k)^{1+q}}\\\nonumber
      &=\int_0^\infty dt
      \,e^{-t}\,\frac{1}{2\pi i} \oint_{|x|=\epsilon} \frac{dx}{x}\,
      e^{x^k/k!} \,
      \frac{x^2 z\, y^k}{\Bigl(x(1+y^k)-ty\Bigr)
        \Bigl(x\bigl(1-y^{2k}(1-z)\bigr) -ty\bigl(1-y^k(1-z)\bigr)\Bigr)}\\\nonumber
      &=\int_0^\infty dt
      \,e^{-t}\,\left(\frac{1-y^k(1-z)}{1-y^{2k}(1-z)}
      \exp{\frac{t^k}{k!}}\left( \frac{y(1-y^k(1-z))}{1-y^{2k}(1-z)} \right)^k
      -\frac{1}{1+y^{k}}
      \exp{\frac{t^k}{k!}}\left( \frac{y}{1+y^{k}} \right)^k
      \right).
    \end{align}
    The integration over $t$ proceeds term-by-term in an expansion of
    the exponentials as was seen in Theorem~\ref{Thmbypolys}. In
    going from the penultimate line to the last, the contour integral
    picks-up two residues, one of which produces precisely minus the
    generating function for the zero-short chord configurations,
    i.e., $-F_k(y,0)$ from Theorem~\ref{Thmbypolys}, corresponding to
    the case of zero connected components. Adding this back as a $z^0$
    term, and replacing $y^k \to y$ we obtain the advertised result
    for $C_k(y,z)$.
\end{proof}

\subsection{Asymptotic distributions}

We expect the asymptotic distribution of short chords, i.e.,
$$
\lim_{n\to \infty} \frac{d_{n,\ell}}{{\cal N}_{k,n}},
$$ to be Poisson with mean given by (the large-$n$ limit of)
Theorem~\ref{Thmmean}
\begin{equation}\nonumber
  \lambda =  \lim_{n\to\infty} {kn\choose k}^{-1}n\,(kn-(k-1))
  = k!\,k^{1-k}\,n^{2-k}.
\end{equation}
The distribution of connected components
$$
\lim_{n\to \infty} \frac{c_{n,q}}{{\cal N}_{k,n}},
$$ should have the same distribution. This is because for long
diagrams, most configurations with $\ell$ short chords will also have
$\ell$ trivially connected components. We begin by considering short
chords.

\begin{theorem}
The asymptotic distribution of short chords is Poisson with mean
$\lambda=k!\,k^{1-k}\,n^{2-k}$.
\end{theorem}
\begin{proof}
Using Lemma~\ref{Lemnearly}, we have that 
  \begin{align}\nonumber
  \lim_{n\to \infty} \frac{d_{n,\ell}}{{\cal N}_{k,n}}
  &=[z^\ell]\lim_{n\to \infty}  \sum_{j=0}^n
  (k!)^j\frac{n!}{(n-j)!}\frac{(k(n-j))!}{(kn)!} {kn - j(k-1)\choose j} (z-1)^j\\\nonumber
  &=[z^\ell]\sum_{j=0}^\infty (k!)^jn^j\frac{1}{(kn)^{kj}} \frac{(kn)^j}{j!}(z-1)^j
  =[z^\ell]\sum_{j=0}^\infty \frac{1}{j!}\left(\frac{k!}{k^{k-1}n^{k-2}}\right)^j
  (z-1)^j\\\nonumber
  &=[z^\ell] \sum_{j=0}^\infty \frac{\lambda^j}{j!} (z-1)^j.
  \end{align}
  It follows that the $j^{\text{th}}$ factorial moment is $\lambda^j$,
  and hence the distribution is Poisson.
\end{proof}

\begin{theorem}
The asymptotic distribution of connected components is Poisson with mean
$\lambda=k!\,k^{1-k}\,n^{2-k}$.
\end{theorem}
\begin{proof}
We begin by expanding the generating function given in Theorem~\ref{Thmbycomp}
\begin{align}\nonumber
   &\frac{c_{n,q}}{{\cal N}_{k,n}} =\frac{1}{{\cal N}_{k,n}} [y^nz^q]\sum_{j\geq
      0}\frac{(kj)!}{j!(k!)^j}\,
    y^j\left(\frac{1-y(1-z)}{1-y^2(1-z)}\right)^{kj+1}
    \\\nonumber &=\sum_{j\geq 0}\frac{{\cal N}_{k,j}}{{\cal N}_{k,n}}
       [y^nz^q]\sum_{p,r} y^{j+p+2r}(z-1)^{p+r} {kj+1\choose p}
       {kj+r\choose r}(-1)^r \\\nonumber
       &= \sum_{j\geq 0}\frac{{\cal N}_{k,j}}{{\cal
           N}_{k,n}}[y^nz^q]\sum_{\ell,r} y^{j+\ell+r}(z-1)^{\ell} {kj+1\choose
         \ell-r} {kj+r\choose r}(-1)^r\\\nonumber
       &= [z^q]\sum_{\ell,r} \frac{{\cal N}_{k,n-\ell-r}}{{\cal
           N}_{k,n}}(z-1)^{\ell} {k(n-\ell-r)+1\choose
         \ell-r} {k(n-\ell-r)+r\choose r}(-1)^r.
 \end{align}
We now take the $n\to\infty$ limit
\begin{align}\nonumber
    &\lim_{n\to\infty}\frac{c_{n,q}}{{\cal N}_{k,n}} =[z^q]\sum_{\ell,r}
    \frac{(k!)^{\ell+r}}{k^{k(\ell+r)}n^{(k-1)(\ell+r)}}\,
    (z-1)^{\ell}\,
    \frac{(kn)^{\ell-r}}{(\ell-r)!}\,
    \frac{(kn)^r}{r!}\,
    (-1)^r\\\nonumber
    &= [z^q]\sum_{\ell,r}
    \frac{(-1)^r}{\ell!} {\ell\choose r} \left(\frac{k!}{k^k\,n^{k-1}}\right)^r
    (z-1)^\ell \, \lambda^\ell = [z^q]\sum_\ell \left(1-\frac{k!}{k^k\,n^{k-1}}\right)^\ell
    \frac{(z-1)^\ell\,\lambda^\ell}{\ell!}\\\nonumber
    &\simeq [z^q]\sum_\ell
        \frac{(z-1)^\ell\,\lambda^\ell}{\ell!}.
\end{align}
It follows that the $j^{\text{th}}$ factorial moment is $\lambda^j$,
and hence the distribution is Poisson.
\end{proof}
  
\section{Non-crossing configurations}\label{Secnc}

For the case of linear chord diagrams the enumeration of so-called
\emph{non-crossing configurations}, where no two chords cross each
other, is by now standard combinatorical lore. For the sake of
completeness, and in order to motivate the case for general $k$, we
repeat the main elements of the arguments given by Kreweras and
Poupard~\cite{KP}, who established a bijection with Dyck paths. The
mapping is as follows: traversing the path of length $2n$ from left to
right, we map the start of a chord to an up step $(0,+1)$, and the end
of a chord with a down step $(+1,0)$.

\begin{figure}[H]
\begin{center}
\includegraphics[bb=0 0 209 76, height=0.95in]{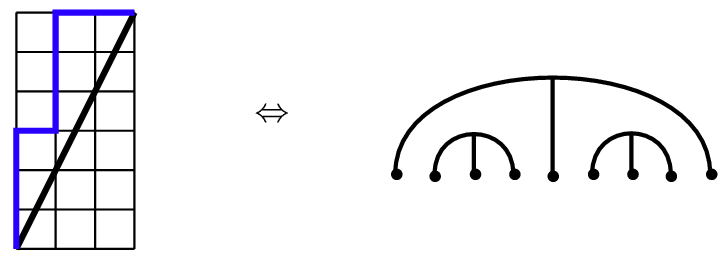}
\end{center}
\caption{The bijection between non-crossing configurations and lattice
  paths for the case of $k=3$.}
\label{Figlattice}
\end{figure}

To establish the mapping in the
other direction, we associate consecutive up steps with the starting
vertices of successively nested chords. It is clear that there are $n$
up steps and $n$ down steps, and that the first step is always an up
step. A short chord is mapped to a peak, i.e., to an up step
immediately followed by a down step. Therefore the Narayana numbers
${n \choose \ell}{n\choose\ell-1}/n$ give the number of non-crossing
configurations with exactly $\ell$ short chords, and the Catalan
numbers ${2n\choose n}/(n+1)$ count the total number of non-crossing
configurations. The bijection to lattice paths can be extended for
general $k$, to paths which begin and end on the line $y=(k-1)x$. A
short chord, traversed left to right, is represented by $k-1$ up steps
followed by a single down step, see
Figure~\ref{Figlattice}\footnote{There is an additional bijection for
  the case of $k=3$ to (rooted) non-crossing trees \cite{N}, where the
  (non-root) vertices correspond to chords, their level to the nesting
  level of the chords, and the leaves to short chords, cf.\ \seqnum{A091320}.}.

In order to generalize the counting to linear $k$-chord diagrams we establish the
following recurrence.
\begin{theorem}\label{Thmnc}
  The number $T_{m,\ell}$ of non-crossing linear $k$-chord diagrams of
  length $km$ with exactly $\ell$ short chords, obeys the following
  recurrence relation
  \begin{equation}\nonumber
  T_{m+1,\ell} = \sum_{\substack{\sum m_i =m\\\nonumber\sum \ell_i=\ell}}\,
  \prod_{i=1}^k T_{m_i,\ell_i} - T_{m,\ell} + T_{m,\ell-1},
  \qquad T_{0,0}=1.
  \end{equation}
\end{theorem}
\begin{proof}
  The proof proceeds diagrammatically, see
  Figure~\ref{Figncrecurrence}. The first term accounts for all
  possible nestings of smaller non-crossing diagrams in the $k-1$
  arches, and also to the left, of an additional $k$-chord. This term
  counts one set of configurations incorrectly, which are those
  pictured on the right in Figure~\ref{Figncrecurrence}. When the
  arches are empty, corresponding to $m_{i> 1}=0,\ell_{i> 1}=0$, the
  additional $k$-chord is also a short chord. Thus this set of
  configurations must be subtracted, and hence the second term in the
  recurrence relation. The third term adds back the correction.
\end{proof}
\begin{figure}[H]
\begin{center}
\includegraphics[bb=0 0 500 76, height=0.75in]{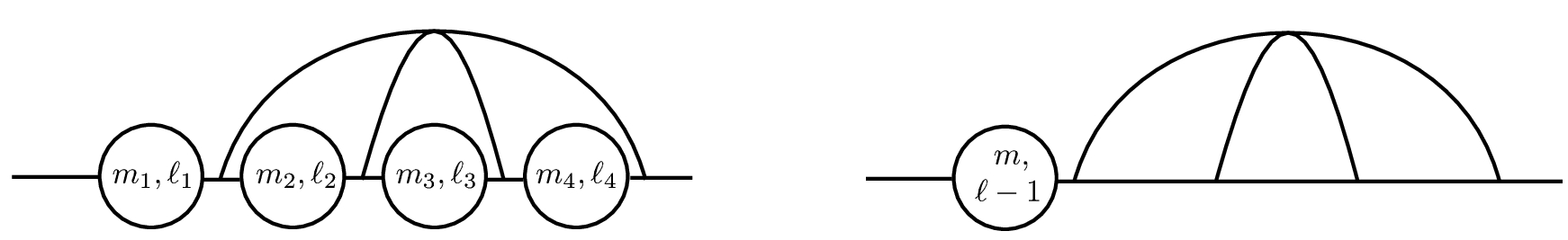}
\end{center}
\caption{Configurations of non-crossing linear $k$-chord diagrams for
  the case of $k=4$. The circles represent all non-crossing linear
  $4$-chord diagrams with $m_i$ $4$-chords, $\ell_i$ of which are
  short chords.}
\label{Figncrecurrence}
\end{figure}
\begin{corollary}\label{CorT}
  The generating function for the $T_{m,\ell}$ is $ T(x,y) =
  \sum_{m,\ell\geq 0} T_{m,\ell}\,x^my^{\ell}$ and obeys
  \begin{equation}\nonumber
    T(x,y) -1 = x T(x,y)^k -x(1-y) T(x,y).
      \end{equation}
\end{corollary}
\begin{proof}
This is shown by multiplying the recurrence relation of
Theorem~\ref{Thmnc} by $x^{m+1}y^\ell$ and summing over $m$ and
$\ell$.
\end{proof}
\begin{corollary}
The total number of of non-crossing linear $k$-chord diagrams of length $km$
is given by the Fuss-Catalan number
\begin{equation}\nonumber
T_m=  \sum_{\ell= 1}^m T_{m,\ell} = \frac{1}{(k-1)m+1}{km\choose m}.
  \end{equation}
\end{corollary}
\begin{proof}
This can be established via Lagrange inversion on Corollary~\ref{CorT}, with
$y=1$. We have that
\begin{align}\nonumber
    &x=\frac{T(x,1)-1}{T(x,1)^k}=f(T(x,1)),\\\nonumber
    &g_m=\lim_{w\to 1} \frac{d^{m-1}}{dw^{m-1}}
  \left(\frac{w-1}{f(w)-f(1)}\right)^m =
  \lim_{w\to 1} \frac{d^{m-1}w^{km}}{dw^{m-1}}
  =\frac{(km)!}{((k-1)m+1)!},\\\nonumber
  &T(x,1) = 1 + \sum_{m>0} g_m \frac{x^m}{m!} =
  1+ \sum_{m>0} \frac{x^m}{(k-1)m+1}{km\choose m}.
\end{align}
The Fuss-Catalan numbers appear as \seqnum{A062993} in the OEIS.
\end{proof}

Kreweras and Poupard~\cite{KP} also give a closed expression for the
number $d_{n,\ell,m}$ of linear chord diagrams on the path of length
$2n$ with $m$ non-crossing chords and $\ell$ short chords $$
d_{n,\ell,m} = \frac{2n-2m+1}{m}{m\choose \ell}{2n-m\choose \ell
  -1}\,d_{n-m,0} ~\Leftrightarrow~ k=2.$$ They obtain this by
considering disconnected regions consisting solely of non-crossing
chords. They note that upon removing these regions, and joining the
remaining components, one is necessarily left with a zero-short-chord
configuration of length $2n-2m$; hence the appearance of the number of
such configurations, i.e., $d_{n-m,0}$. For a general value of $k$, we
have the following theorem.
\begin{theorem}
The number $d_{n,\ell,m}$ of linear $k$-chord diagrams of length $kn$ with
exactly $\ell$ short chords, and exactly $m$ non-crossing chords is
given by
  \begin{equation}\nonumber
d_{n,\ell,m}=[x^my^\ell]\,T(x,y)^{kn-km+1}\,d_{n-m,0},
  \end{equation}
where $d_{n-m,0}$ is the number of zero-short-chord configurations.
\end{theorem}
\begin{proof}
  We consider a general linear $k$-chord diagram on the path of length
  $kn$. Let there be $p$ disjoint regions, each consisting solely of
  non-crossing $k$-chords. When we remove these regions, we are left
  with a path of length $kn-km$ where $m$ is the number of
  non-crossing $k$-chords. There are thus ${kn-km+1\choose p}$
  distinct ways of placing the $p$ regions. We need to sum over all
  possible sizes of these regions, and also over all possible
  distributions of the $\ell$ short chords amongst the $p$
  regions. These sums take place automatically using the symbolic
  method by adding a factor of $(T(x,y)-1)^p$, where we subtract $1$
  because the regions cannot be empty. Finally, we multiply by
  $d_{n-m,0}$, because the configuration on the path of length $kn-km$
  is necessarily one with no short chords. We therefore have that
  \begin{equation}\nonumber
    d_{n,\ell,m}=d_{n-m,0}\sum_{p}{kn-km+1\choose p} [x^my^\ell]
    \left(T(x,y)-1\right)^p
    = [x^my^\ell]\,T(x,y)^{kn-km+1}\,d_{n-m,0}.
     \end{equation}
\end{proof}
\begin{table}[H]
\begin{center}
  \begin{tabular}{c|lllllll}
  $n$ \textbackslash$\ell$& 1& 2& 3& 4& 5& 6& 7\\
    \hline  
1& 1\\
2& 2& 1\\
3& 4& 7& 1\\
4& 8& 30& 16& 1\\
5& 16& 104& 122& 30& 1\\
6& 32& 320& 660& 365& 50& 1\\
7& 64& 912& 2920& 2875& 903& 77& 1\\
\end{tabular}
\end{center}
\caption{The number $T_{n,\ell}$ of non-crossing linear $k$-chord
  diagrams with exactly $\ell$ short chords for the case $k=3$. OEIS
  sequence \seqnum{A091320}; the $k=4$ case is \seqnum{A334062} and
  the $k=5$ case is \seqnum{A334063}.}
\end{table}

\subsection{Asymptotic distribution of short chords}

In this section we consider the asymptotic distribution of short chords
amongst non-crossing configurations, i.e.,
$$
\lim_{n\to\infty}\frac{T_{n,\ell}}{T_n}.
$$ It has been established that the Narayana numbers (i.e., the $k=2$
case) are asymptotically normally distributed, cf.\ \cite{FR}, with
mean $\mu = n/2$ and variance $\sigma^2 = n/8$. In this section we
appeal to the methods of Flajolet and Noy~\cite[Theorem 5]{FN} to
establish the following generalization.
\begin{theorem}
  The numbers $T_{n,\ell}$ of non-crossing linear $k$-chord diagrams
  of length $kn$, with exactly $\ell$ short chords, are asymptotically
  normally distributed with mean $\mu$ and variance $\sigma^2$ given
  by
   \begin{equation}\nonumber
     \mu = \left(\frac{k-1}{k}\right)^{k-1}n,\qquad
     \sigma^2 = \left(\frac{k-1}{k}\right)^{2k}\frac{k}{(k-1)^2}
     \left( 1-2k+(k-1)\left(\frac{k}{k-1}\right)^k\right)n.
     \end{equation}
\end{theorem}
\begin{proof}
The methods used by Flajolet and Noy~\cite[Theorem 5]{FN} extend the
analytic combinatorics of implicitly defined generating functions, as
treated in Flajolet and Sedgewick~\cite[Section VI.7]{FS}, to the case
of bivariate generating functions. In the univariate case
(i.e., setting $y=1$ in $T(x,y)$), it is straightforward to establish
an asymptotic expansion
\begin{equation}\nonumber
T(x,1) = d_0 + d_1\sqrt{1-x/\rho} + {\cal O}(1-x/\rho),
 \end{equation}
which then implies the asymptotic growth
\begin{equation}\nonumber
[x^n]T(x,1)\sim \gamma \frac{\rho^{-n}}{\sqrt{\pi n^3}} \left(1+{\cal
  O}(n^{-1})\right), \quad \rho = \frac{(k-1)^{k-1}}{k^k},\quad
\gamma = \sqrt{\frac{k}{2(k-1)^3}}.
 \end{equation}
The method of Flajolet and Noy is to extend this to the bivariate case
by considering $y$ as a parameter. We begin by expressing the recurrence
relation of Theorem~\ref{Thmnc} as follows:
 \begin{equation}\nonumber
    T(x,y)=x\,\phi\left(T(x,y)\right)
   \Rightarrow \phi(u)=(1+u)^k-(1-y)(1+u).
  \end{equation}
We are then tasked with solving the so-called \emph{characteristic equation}
   \begin{equation}\nonumber
     \phi(\tau(y))-\tau(y) \phi'(\tau(y))=0,
   \end{equation}
which in our case is
   \begin{equation}\nonumber
     (1+\tau)^k-(1-y)(1+\tau) -\tau \left( k(1+\tau)^{k-1} -(1-y) \right) =0.
   \end{equation}
Solving this equation in an expansion of $\tau(y)$ about $y=1$, we
find
   \begin{equation}\nonumber
     \tau(y) = \frac{1}{k-1}+
     \frac{k}{(k-1)^2}\left(\frac{k-1}{k}\right)^k(y-1)
     -\frac{k}{(k-1)^2} \left(\frac{k-1}{k}\right)^{2k}(y-1)^2
     +{\cal O}\left((y-1)^3\right),
   \end{equation}
which further implies an expansion for 
   \begin{equation}\nonumber
     \rho(y) = \frac{\tau(y)}{\phi\left(\tau(y)\right)}.
   \end{equation}
   Flajolet and Noy establish that this implies the following asymptotic growth
   \begin{equation}\nonumber
[x^n]T(x,y) = \gamma(y) \,\rho(y)^{-n} \left(1+{\cal O}(n^{-1/2})\right),
   \end{equation}
   where $\gamma(y)$ is an analytic function of $y$. The asymptotic
   normality is then established through results due to Bender and
   Richmond~\cite{BR}. The probability generating function for the
   distribution is given by
    \begin{equation}\nonumber
      \frac{ [x^n]T(x,y)}{[x^n]T(x,1)} = \frac{\gamma(y)}{\gamma(1)}
      \left( \frac{\rho(y)}{\rho(1)}  \right)^{-n} \left(1+{\cal O}(n^{-1/2})\right).
    \end{equation}
    The mean and variance are computed in the usual way
   \begin{equation}\nonumber
     \mu = -\frac{\rho'(1)}{\rho(1)}, \qquad
     \sigma^2 =  -\frac{\rho''(1)}{\rho(1)} -\frac{\rho'(1)}{\rho(1)}
     + \left(\frac{\rho'(1)}{\rho(1)}\right)^2,
   \end{equation}
   and this results in
   \begin{equation}\nonumber
     \mu = \left(\frac{k-1}{k}\right)^{k-1}n,\qquad
     \sigma^2 = \left(\frac{k-1}{k}\right)^{2k}\frac{k}{(k-1)^2}
     \left( 1-2k+(k-1)\left(\frac{k}{k-1}\right)^k\right)n.
     \end{equation}
\end{proof}

\begin{corollary}
The mean $\mu$ and variance $\sigma^2$ are given by the following
expressions in the $k\to\infty$ limit
\begin{equation}\nonumber
  \mu = \frac{n}{e},\qquad
  \sigma^2 = \frac{e-2}{e^2}\,n.
  \end{equation}
\end{corollary}

\section{Applications to generalized games of memory}

In the game of memory, $n$ distinct pairs of cards are placed in an
array. The present author~\cite{DY1,DY2} has enumerated configurations
for $2\times n$ rectangular arrays in which exactly $\ell$ of the
pairs are found side-by-side, or over top of one another, thus forming
$1\times 2$ or $2\times 1$ dominoes. The enumeration of these
configurations always carries a factor of $n!$, which counts the
orderings of the $n$ distinguishable pairs. It is therefore easier to
drop this factor, and thus treat the pairs as indistinguishable. For
the case of $1 \times 2n$ arrays the configuration of the cards is
then in one-to-one correspondence with linear chord diagrams, where
the dominoes correspond to short chords.

More generally, the array which the cards are placed on can be
specified by a graph $G$ on $2n$ vertices, representing the positions
of the cards, and such that an edge of $G$ indicates adjacency in the
array --- for example, a rectangular array would be specified by a grid
graph of the same dimension. A matched pair of cards occupying two
adjacent vertices of $G$ then constitutes a domino. The present
author~\cite{DY1} obtained the mean for the distribution of dominoes
for general arrays was obtained in terms of the number of edges of the
graph $G$.

The results of the present paper lend themselves to a generalized game
of memory in which $n$ distinct sets of $k$ matched cards are placed
in an array specified by a graph $G$ on $kn$ vertices. The notion of a
domino then generalizes to that of a polyomino, see
Figure~\ref{Figk3n3example}. The results of the previous sections may
be interpreted as pertaining to this generalized game, where the array
is a single row of length $kn$, and so $G$ is the path on $kn$
vertices.
\begin{definition} 
  A \emph{polyomino} is \emph{a set of $k$ matching cards which occupy a
    connected subgraph of $G$ with $k$ vertices}.
\end{definition}
\begin{figure}[H]
\begin{center}
\includegraphics[bb=0 0 74 74, height=0.75in]{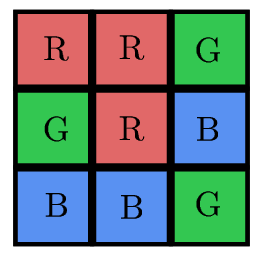}
\end{center}
\caption{A configuration with a single polyomino (the red cards marked ``R'') in a
  generalized game of memory with $k=3$ and $n=3$, played on a square
  array.}
\label{Figk3n3example}
\end{figure}
We may employ a slightly more general version of
Theorem~\ref{Thmmean} to compute the mean number of polyominoes.
\begin{theorem}\label{Thmmeanp}
The mean number of polyominoes in a generalized game of memory
played on a graph with $kn$ vertices and $r$ connected
subgraphs, each with $k$ vertices, is given by
\begin{equation}\nonumber
{kn \choose k}^{-1}nr .
\end{equation}
\end{theorem}
\begin{proof}
The proof proceeds through the linearity of expectation. Let the
random variable $X_j$ take the value $1$ when the $j^{\text{th}}$
connected subgraph forms a polyomino matching and $0$ otherwise. Once a
polyomino is thusly placed, by Proposition~\ref{Nnk}, there are
${\cal N}_{k,n-1}$ ways of placing the remaining cards on the $kn-k$
remaining vertices of the graph. Thus $E(X_j) = {\cal N}_{k,n-1}/{\cal
  N}_{k,n}$. We therefore have that $E(\sum_{j=1}^r X_j) =
\sum_{j=1}^r E(X_j) = r \, {\cal N}_{k,n-1}/{\cal N}_{k,n}$.
\end{proof}
It is interesting to consider how this mean scales with $n\to\infty$
for hypercubical grid (or other regular) graphs. For fixed $k$, we
expect the number $r$ of subgraphs on $k$ vertices to scale as $n$. We
thus expect the mean to scale as $n^2/n^k=n^{2-k}$. Indeed, for large
hypercubical grid graphs of dimension $d$, the present
author~\cite[Corollary 2]{DY1} showed that the mean number of dominoes
approaches $d$ (and is hence ${\cal O}(n^0)$). For $k>2$, polyominoes
are instead {\it suppressed} as $n$ grows large.

We may also introduce the notion of a \emph{connected component} of
polyominoes, see Figure~\ref{Figconcomp}.
\begin{definition}\label{Defcomppoly}
  A \emph{connected component} is a \emph{set of polyominoes which occupy a
    single connected subgraph of $G$}.
\end{definition}
\begin{figure}[H]
\begin{center}
\includegraphics[bb=0 0 120 74, height=0.75in]{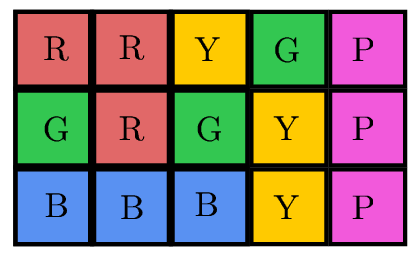}
\end{center}
\caption{A configuration in a generalized game of memory with $k=3$
  and $n=5$ with two connected components: the red ``R'' and blue
  ``B'' cards form one connected component and the pink ``P'' cards
  form another.}
\label{Figconcomp}
\end{figure}
The notion of non-crossing configurations is more subtle to generalize
beyond the case of the linear array, and a precise definition will be
left to further work. The problem of enumerating configurations in
generalized games of memory by polyominoes, connected components, and
(suitably defined) non-crossing arrangements is interesting, and
likely very difficult to obtain exact results for in general.

\bigskip
\hrule
\bigskip

\noindent {\it 2010 Mathematics Subject Classification:} 
Primary 05A15; Secondary 05C70, 60C05. 

\noindent \emph{Keywords:} 
chord diagram, perfect matching. 
\bigskip
\hrule
\bigskip

\noindent (Concerned with sequences 
\seqnum{A062993},
\seqnum{A091320},
\seqnum{A334056},
\seqnum{A334057},
\seqnum{A334058},
\seqnum{A334059},
\seqnum{A334060},
\seqnum{A334061},
\seqnum{A334062}, and
\seqnum{A334063}.)
\bigskip
\hrule
\bigskip

\vspace*{+.1in}
\noindent

\bigskip
\hrule
\bigskip


\end{document}